\documentclass[a4paper]{article}
\usepackage{header_math}
\usepackage{header_article}
\usepackage{comment}
\usepackage{framed}
\usepackage{subcaption}
\usepackage[outline]{contour}
\usepackage{tikz}
\usetikzlibrary{calc}
\usetikzlibrary{patterns}
\usetikzlibrary{decorations.pathmorphing}
\usetikzlibrary{decorations.markings}
\usetikzlibrary{decorations.pathreplacing}

\usepackage{graphicx}

\newcommand{\flipper}{\texttt{flipper}}
\newcommand{\RRz}{\RR_{\geq 0}}
\DeclareMathOperator{\nummarkedpoints}{n}
\let\genus\relax
\DeclareMathOperator{\genus}{g}  

\DeclareMathOperator{\canonical}{\sigma}
\DeclareMathOperator{\logdegree}{dg}

\DeclareMathOperator{\height}{hgt}

\title{The pseudo-Anosov and conjugacy problems are in $\NP \cap \coNP$}
\author{Mark C. Bell\\
University of Illinois\\
\url{mcbell@illinois.edu}}

\begin{document}

\maketitle

\begin{abstract}
For a fixed marked surface $S$, we construct polynomial bounds on the periodic and preperiodic lengths of the maximal splitting sequences of a projectively invariant measured train track.

We give two consequences of these bounds. Firstly, that the problem of deciding whether a mapping class is pseudo-Anosov lies in $\NP$. This is dual to the previously known result that the pseudo-Anosov problem is in $\coNP$. Secondly, that the problem of deciding whether two mapping classes are conjugate lies in $\coNP$. Similarly, this is the dual to the previously known result that the conjugacy problem is in $\NP$.

As usual, in both cases we immediately obtain exponential time solutions to these problems. A version of these algorithms have been implemented as part of \flipper{}.
\end{abstract}

\keywords{mapping class group; pseudo-Anosov; conjugacy; $\NP \cap \coNP$.}

\ccode{37E30, 68W40, 57-04}

\section{Introduction}


Fix $S$ to be a marked surface with at least one marked point (or at least three when the surface has no genus).

Fix $X$ to be a finite generating set of $\Mod(S)$, the \emph{mapping class group} of $S$ (relative to the set of marked points). Let $X^*$ denote the set of all \emph{words} that can be made using the elements of $X$ as letters. We identify a word $h = h_1 \cdots h_k \in X^*$ with the mapping class $h_1 \circ \cdots \circ h_k$ and denote its \emph{length} by $\ell(h) \defeq k$.

Let $\ML(S)$ denote the space of \emph{measured laminations} on $S$ \cite[Section~1.7]{PH} \cite[Section~3]{CB}. Recall that a mapping class $h \in \Mod(S)$ is \emph{pseudo-Anosov} \cite[Expos\`{e}~12]{FLP} \cite[Page~95]{CB} if there is a measured lamination $\calL \in \ML(S)$ which is:
\begin{itemize}
\item \emph{projectively invariant}, that is, $h(\calL) = \lambda \cdot \calL$ for some $\lambda \in \RR$, and
\item \emph{filling}, that is, it assigns positive measure to every essential, simple closed curve.
\end{itemize}
We say that a word is \emph{pseudo-Anosov} if its corresponding mapping class is.

\begin{problem}[The Pseudo-Anosov Problem]
Given a word $h \in X^*$, decide whether $h$ is pseudo-Anosov.
\end{problem}

If $h$ is not pseudo-Anosov then by the Nielsen--Thurston classification \cite[Chapter~13]{FM} it is either \emph{periodic} or \emph{reducible}. There is a polynomial time solution to the word problem for $\Mod(S)$ \cite[Theorem~4.2]{FM} and an upper bound on the order of finite order mapping classes \cite[Theorem~7.5]{FM}. Together these give a polynomial time algorithm to determine whether $h$ is \emph{periodic}. On the other hand, there are several proofs that if $h$ is reducible then there is a reducing curve whose complexity is at most $O(\poly(\ell(h)))$ \cite[Theorem~3.3]{BellReducibility} \cite[Theorem~1.1]{KoberdaMangahas}. Such a curve acts as a certificate of reducibility which can be verified in polynomial time. This shows that the pseudo-Anosov problem is in $\coNP$ \cite[Corollary~3.5]{BellReducibility}.

In Section~\ref{sec:pA} we prove the dual result. Namely, if $h \in X^*$ is pseudo-Anosov then there is a description of a measured lamination $\calL$ such that it can be checked that $\calL$ is projectively invariant and filling in $O(\poly(\ell(h)))$ time. That is, if a word is pseudo-Anosov then there is a small proof of this, or equivalently:
\begin{restate}{Corollary}{cor:pA_NP}
The pseudo-Anosov problem lies in $\NP$. \qed
\end{restate}

In the second half of this paper, we consider the \emph{conjugacy problem}. Here we write $g \equiv h$ if two words $g, h \in X^*$ represent the same mapping class.

\begin{problem}[The Conjugacy Problem]
Given words $g, h \in X^*$, decide whether $g$ and $h$ represent \emph{conjugate} mapping classes. That is, decide whether there is there a word $f \in X^*$ such that $f g f^{-1} \equiv h$.
\end{problem}

Theorem~B of \cite{jT} states that there there is a constant $K = K(S, X)$ such that if two words $g, h \in X^*$ represent conjugate mapping classes then there is a word $f \in X^*$ such that
\[ h \equiv f g f^{-1} \inlineand \ell(f) \leq K \cdot (\ell(g) + \ell(h)). \]
Such a word $f$ acts as a certificate that $g$ and $h$ are conjugate and, as it is sufficiently small, this can be verified in polynomial time \cite{MosherAutomatic}. This shows that the conjugacy problem is in $\NP$.

In Section~\ref{sec:conjugacy} we use the results developed in Section~\ref{sec:train_tracks} to again show the dual result. Namely, if a pair of words do not represent conjugate mapping classes then there is a short proof of this, or equivalently:
\begin{restate}{Corollary}{cor:conjugacy_coNP}
The conjugacy problem lies in $\coNP$. \qed
\end{restate}

Finally, in Section~\ref{sec:implementation} we discuss some important implementation details of these algorithms.

\subsection{Model of computation}

The operations that we will describe will involve manipulating reasonably large numbers. We will take care to include the difficultly of performing these in our analysis of the algorithm.

\begin{definition}
\label{def:bounded}
An integer $x$ is \emph{$k$--bounded} if it has at most $k$ digits, that is, if it can be represented by $O(k)$ bits.
\end{definition}

As part of our model of computation, we will assume that if $x$ and $y$ are $k$--bounded then:
\begin{itemize}
\item $\sgn(x)$ can be computed in $O(1)$ operations,
\item $x \pm y$ is $(k + 1)$--bounded and can be computed in $O(k)$ operations, and
\item $xy$ is $2k$--bounded and can be computed in $O(k^2)$ operations.
\end{itemize}
However, for ease of argument, we will assume a model of computation in which there is no cost associated to accessing variables.

\subsection{Models of \texorpdfstring{$\ML(S)$}{ML(S)} and \texorpdfstring{$\Mod(S)$}{Mod(S)}}
\label{sub:model_ML}

Fix $\calT$ to be an (ideal) triangulation of $S$ with (ordered) edges $e_1, \ldots, e_\zeta$. The triangulation $\calT$ provides a coordinate system on $\ML(S)$ where the lamination $\calL \in \ML(S)$ is represented by the \emph{edge vector}:
\[\calT(\calL) \defeq
\left(
\begin{array}{c}
\calL(e_1) \\
\vdots \\
\calL(e_\zeta) \\
\end{array}
\right) \in \RRz^\zeta
\]
Although distinct laminations have distinct coordinates, not all coordinates correspond to measured laminations. In fact $v_1, \ldots, v_\zeta$ corresponds to a measured lamination if and only if:
\begin{itemize}
\item $\sum v_i > 0$,
\item for each triangle $a, b, c$ it satisfies the triangle inequality $v_a + v_b \geq v_c$, and
\item for each vertex $v$ there is an incident triangle $a, b, c$ such that $a \cap b = v$ and $v_a + v_b = v_c$.
\end{itemize}
We express this requirement as a linear programming problem:

\begin{lemma}
\label{lem:is_multicurve_LP}
For each triangulation $\calT$, there are $O(1)$--bounded $\zeta \times 3\zeta$ matrices $F_1, \ldots, F_k$ such that $v \in \RRz^\zeta$ is in the image of $\calT(\cdot)$ if and only if $v \neq 0$ and
\[ F_i \cdot v \geq 0 \]
for some $i$. \qed
\end{lemma}

Let $G$ denote the \emph{graph of (labelled) triangulations} of $S$. This is a graph whose vertices are triangulations of $S$ and two are connected via an edge of length one if and only if they differ by a single \emph{flip}, as shown in Figure~\ref{fig:flip}, or a \emph{reordering} or the edges. As $G$ and $\Mod(S)$ are quasi-isometric \cite[Proposition~8.19]{BridsonHaefliger}, we may think of a word $h \in X^*$ as a path $p$ in $G$ from $\calT$ to $h(\calT)$. Moreover we may choose this path such that $\ell(p) \in O(\ell(h))$.

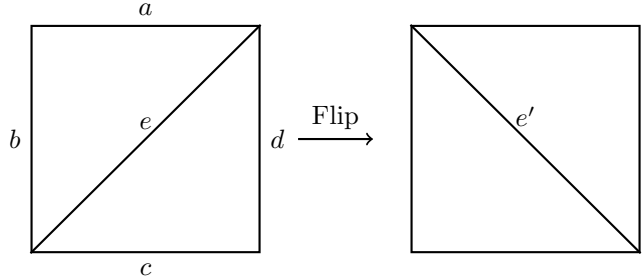
\begin{figure}[ht]
\centering
\begin{tikzpicture}[scale=2,thick]

\node (rect) at (-1.5, 0) [draw,minimum width=3cm,minimum height=3cm] {};
\node (rect2) at (1, 0) [draw,minimum width=3cm,minimum height=3cm] {};

\draw (rect.south west) -- node [above] {$e$} (rect.north east);
\draw (rect2.north west) -- node [above, yshift=1] {$e'$} (rect2.south east);

\node (a) at (rect.north) [anchor=south] {$a$};
\node (b) at (rect.west) [anchor=east] {$b$};
\node (c) at (rect.south) [anchor=north] {$c$};
\node (d) at (rect.east) [anchor=west] {$d$};

\draw [thick,->] ($(rect.east)!0.25!(rect2.west)$) -- node[above] {Flip} ($(rect.east)!0.75!(rect2.west)$);
\end{tikzpicture}
\caption{Flipping an edge of a triangulation.}
\label{fig:flip}
\end{figure}

The edge vector of a lamination is clearly well behaved under a reordering but it is also well behaved under a flip.

\begin{proposition}[{\cite[Page~30]{MosherFoliations}}]
\label{prop:flip_intersection_laminations}
Suppose that $\calL$ is a measured lamination and $e$ is a flippable edge of a triangulation $\calT$ as shown in Figure~\ref{fig:flip} then
\[ \calL(e') = \max(\calL(a) + \calL(c), \calL(b) + \calL(d)) - \calL(e). \inlineQED \]
\end{proposition}

Thus there is a piecewise-linear function which transforms $\calT$ coordinates to $\calT'$ coordinates. We express this piecewise-linear function using two collections of matrices $\{A_i\}$ and $\{B_i\}$.

\begin{lemma}
\label{lem:encoding_LP}
Suppose that $p$ is a path from $\calT$ to $\calT'$. There are matrices $\{A_i\}$ and $\{B_i\}$, as described in \cite[Section~2.2]{BellReducibility}, such that:
\begin{enumerate}
\item Each $A_i$ and $B_i$ is $\ell(p)$--bounded.
\item Each $B_i$ has $O(\ell(p))$ rows.
\item For each measured lamination $\calL \in \ML(S)$ we have that $B_i \cdot \calT(\calL) \geq 0$ for some $i$.
\item For each measured lamination $\calL \in \ML(S)$, if $B_i \cdot \calT(\calL) \geq 0$ then $\calT'(\calL) = A_i \cdot \calT(\calL)$. \qed
\end{enumerate}
\end{lemma}

\begin{remark}
\label{rem:compute_image}
Suppose that $p$ is a path from $\calT$ to $\calT'$ and that $\{A_i\}$ and $\{B_i\}$ are the matrices of Lemma~\ref{lem:encoding_LP}. If $v$ is a vector in which each entry has at most $d$ digits then by considering each flip in turn we can find $B_i$ such that $B_i \cdot v \geq 0$ in at most $O(d \ell(p) + \ell(p)^2)$ operations. Moreover each entry of $A_i \cdot v$ has at most $d + \ell(p)$ digits.
\end{remark}

\section{Train tracks}
\label{sec:train_tracks}

We begin with a discussion of train tracks \cite{PH}. These will be crucial for determining if a measured lamination is filling. To ease notation, throughout this section we will assume that $h \in X^*$ is a fixed word and that $\calL \in \ML(S)$ is a measured lamination such that:
\[ h(\calL) = \lambda \cdot \calL. \]

\begin{definition}[{\cite[Section~1.1]{PH}}]
A \emph{measured train track (representing $\calL$)} is a pair $T = (\tau, \mu)$ consisting of:
\begin{itemize}
\item a \emph{train track} $\tau$, that is, a trivalent graph on $S$ (whose vertices we refer to as \emph{switches} and edges we refer to as \emph{branches}) with a well defined tangency at each switch such that there is no switch with all branches emanating from it in the same direction and no complementary region of $\tau$ is a nullogon, monogon, bigon, once-marked nullogon or annulus, and
\item a \emph{transverse measure} $\mu$ such that there is a smooth map $\phi \from S \to S$, isotopic to the identity, such that:
\begin{itemize}
\item $\phi(\calL) = \tau$,
\item $\phi|_{\calL} \from \calL \to \tau$ is a submersion, and
\item $\mu = \calL \circ \phi^{-1}$
\end{itemize}
\end{itemize}
\end{definition}

\begin{definition}
If $T = (\tau, \mu)$ is a measured train track then we may \emph{split} along along one of its branches $e$ to obtain a new measured train track $T' = (\tau', \mu')$ as shown in Figure~\ref{fig:splitting_branch} \cite[Page~119]{PH}.
\end{definition}

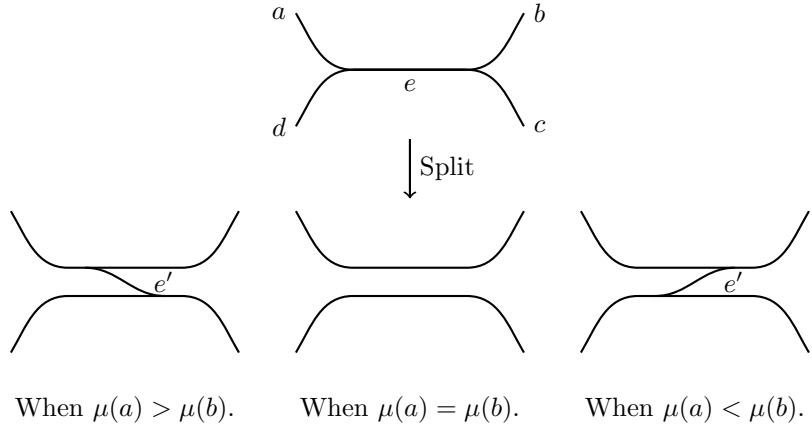
\begin{figure}[ht]
\centering
\begin{tikzpicture}[scale=0.75,thick]
\coordinate (A) at (-1,0);
\coordinate (B) at (1,0);
\coordinate (W) at (-2,1);
\coordinate (X) at (2,1);
\coordinate (Y) at (-2,-1);
\coordinate (Z) at (2,-1);
\coordinate (x) at ($(A)!0.15!(B)$);
\coordinate (y) at ($(A)!0.85!(B)$);
\coordinate (z) at ($(A)!0.5!(B)$);

\draw (A) -- node[below] {$e$} (B);
\draw (W) node [left] {$a$} to [out=-60,in=180] (A) to [out=0,in=180] (B) to [out=0,in=180+60] (X) node [right] {$b$};
\draw (Y) node [left] {$d$} to [out=+60,in=180] (A) to [out=0,in=180] (B) to [out=0,in=180-60] (Z) node [right] {$c$};

\coordinate (ysep) at (0,0.5);
\foreach \x / \y in {-5/-4, 0/-4, 5/-4}
{
  \draw
    ($(W) + (ysep) + (\x, \y)$) to [out=-60,in=180] 
    ($(A) + (ysep) + (\x, \y)$) to [out=0,in=180] 
    ($(B) + (ysep) + (\x, \y)$) to [out=0,in=180+60]
    ($(X) + (ysep) + (\x, \y)$);
  \draw 
    ($(Y) + (\x, \y)$) to [out=+60,in=180] 
    ($(A) + (\x, \y)$) to [out=0,in=180] 
    ($(B) + (\x, \y)$) to [out=0,in=180-60] 
    ($(Z) + (\x, \y)$);
}

\draw ($(x) + (ysep) + (-5,-4)$) to [out=0, in=180] node [right, shift={(0.25,0)}] {$e'$} ($(y) + (-5,-4)$);
\draw ($(x) + (5,-4)$) to [out=0, in=180] node [right, shift={(0.25,0)}] {$e'$} ($(y) + (ysep) + (5,-4)$);
\node at ($(z) + (-5,-6)$) {When $\mu(a) > \mu(b)$.};
\node at ($(z) + (0,-6)$) {When $\mu(a) = \mu(b)$.};
\node at ($(z) + (5,-6)$) {When $\mu(a) < \mu(b)$.};

\coordinate (S) at ($(Y)!0.5!(Z)$);
\coordinate (E) at ($(S)+(0,-1.5)$);
\draw [->] ($(S)!0.15!(E)$) -- node [right] {Split} ($(E)!0.15!(S)$);
\end{tikzpicture}
\caption{The possibilities for splitting a branch $e$. In either case $\mu(e')
= |\mu(a) - \mu(b)| = |\mu(c) - \mu(d)|$ \cite[Figure~2.1.2]{PH}.}
\label{fig:splitting_branch}
\end{figure}

There are many different measured train tracks that represent $\calL$. Again we represent these via a simplicial graph $\GG$, the \emph{graph of measured train tracks (representing $\calL$)}. This is a graph whose vertices are measured train tracks (representing $\calL$) and there is an (unoriented) edge from $T$ to $T'$ if and only if $T'$ can be obtained by splitting some collection of the branches of $T$, each of which has the same transverse measure. This graph is connected \cite[Theorem~2.8.5]{PH} and we write $d(T, T')$ for the distance between two measured train tracks when each edge of this graph is assigned length one.

\subsection{Maximal splittings and the axis}

\begin{definition}[{\cite[Section~2]{A}}]
The \emph{maximal splitting} of a measured train track $T$ is the measured train track $s(T)$ obtained by simultaneously splitting all branches of maximal measure. 
\end{definition}

\begin{lemma}
\label{lem:non_expansive}
If $T$ and $T'$ are measured train tracks (representing $\calL$) then
\[ d(s(T), s(T')) \leq d(T, T'). \]
That is, $s$ is a non-expansive map.
\end{lemma}

\begin{proof}
Begin by considering the case in which $d(T, T') = 1$. Without loss of generality, there is a subset $B$ of the branches of $T$ such that by splitting the branches in $B$ we obtain $T'$ and each branch in $B$ has the same transverse measure. Let $M$ be the subset of the branches of $T$ with maximal transverse measure. There are three possible cases to consider:
\begin{enumerate}
\item If $B = M$ then $s(T) = T'$ and so $d(s(T), s(T')) = d(T', s(T')) \leq 1$.
\item If $B \subset M$ then the branches in $M - B$ also appear in $T'$. It is these branches that are split when maximally splitting $T'$ and so $s(T) = s(T')$. Hence $d(s(T), s(T')) = 0$.
\item If $B \nsubseteq M$ then the branches in $B$ also appear in $s(T)$. Splitting these branches results in $s(T')$ as $M$ and $B$ lie in disjoint open sets on $S$. Hence $d(s(T), s(T')) = 1$.
\end{enumerate}
In any case $d(s(T), s(T')) \leq 1$ and so the result holds by induction on $d(T, T')$.
\end{proof}

Agol showed that the sequence of train tracks obtained by repeatedly performing maximal splittings is eventually periodic (up to rescaling).

\begin{theorem}[{\cite[Theorem~3.5]{A}}]
\label{thrm:maximal_splits_preperiodic}
If $T$ is a measured train track then there exists $m, n \in \NN$ such that $s^{m+n}(T) = \widehat{h}(s^n(T))$ where
\[ \widehat{h}(\tau, \mu) \defeq (h(\tau), \mu / \lambda). \inlineQED \]
\end{theorem}

We refer to the smallest such $m$ and $n$ as the \emph{periodic} and \emph{preperiodic lengths} of $T$ respectively. We note that $m$ depends only on $\calL$ and is independent of $T$.

\begin{definition}
The \emph{axis} of $\calL$ is the bi-infinite sequence of measured train tracks $A = A(\calL) \defeq \{\mathbf{T}_i\}_{i=-\infty}^\infty$ such that 
\[ s(\mathbf{T}_i) = \mathbf{T}_{i+1} \inlineand \widehat{h}(\mathbf{T}_i) = \mathbf{T}_{i+m}. \]
\end{definition}

The measured train tracks on the axis are useful as you can determine if $\calL$ is filling purely from the combinatorics of their underlying train tracks.

\begin{definition}
A measured train track $T = (\tau, \mu)$ is \emph{filling} if every complementary region of $\tau$ is either a disk or a once-marked disk.
\end{definition}

\begin{lemma}
\label{lem:filling_axis}
The measured lamination $\calL$ is filling if and only if $\mathbf{T}_i$ is.
\end{lemma}

\begin{proof}
If $\mathbf{T}_i = (\tau, \mu)$ is not filling then there is an essential, simple closed curve $\gamma$ in the complement of $\tau$. Therefore $\calL(\gamma) = 0$ and so $\calL$ is not filling.

Conversely, if $\calL$ is not filling then there is an essential, simple closed curve $\gamma$ such that $\calL(\gamma) = 0$. There is a measured train track $T = (\tau, \mu)$, representing $\calL$, such that $\tau$ and $N(\gamma)$ are disjoint. By Theorem~\ref{thrm:maximal_splits_preperiodic} there are $j$ and $k$ such that 
$\widehat{h}^k(s^j(T)) = \mathbf{T}_i$. As $T$ is not filling and this is preserved by both maximal splittings and homeomorphisms, $\mathbf{T}_i$ is not filling either.
\end{proof}

\subsection{Getting to the axis}

We use Lemma~\ref{lem:filling_axis} to decide whether $\calL$ is filling by constructing a measured train track on the axis $A$. To do this we give an upper bound the periodic and preperiodic lengths of $T$.

\begin{theorem}
\label{thrm:periodic_bound}
The periodic length of $T$ is at most $d(T, \widehat{h}(T))$.
\end{theorem}

\begin{proof} 
Suppose that 
\[ p \defeq T_0, T_1, \ldots, T_k \]
is a \emph{path} from $T$ to $\widehat{h}(T)$, that is, a sequence such that $d(T_i, T_{i+1}) = 1$. Let $T_i^j \defeq s^{j}(T_i)$. Note that for each $j$:
\begin{itemize}
\item $\widehat{h}(T_0^j) = T_k^j$, and
\item $p_j \defeq T_0^{j}, \ldots, T_k^{j}$ is also a path as $s$ is non-expansive.
\end{itemize}

When $j$ is sufficiently large, by Theorem~\ref{thrm:maximal_splits_preperiodic} each $T_i^j$ lies on the axis $A$ and therefore so does $p_j$. The endpoints of $p_j$ must be $\mathbf{T}_{k'}$ and $\mathbf{T}_{k' + m}$ and so we conclude that $m \leq k$.

As this holds for every such path $p$, we have that the periodic length of $T$ is at most $d(T, \widehat{h}(T))$.
\end{proof}

To bound the preperiodic length, the following theorem is key. In particular, note that the constant depends only on the underlying surface and is independent of the chosen mapping class $h$ and measured lamination $\calL$.

\begin{theorem}[{\cite[Theorem~B]{jT}}]  
\label{thrm:linear_conjugator_bound}
There is a constant $K = K(S)$ such that if $T$ and $T'$ are measured train tracks then there is a measured train track $T''$ in the orbit of $T'$ under $h$ such that
\[ d(T, T'') \leq K(d(T, \widehat{h}(T)) + d(T', \widehat{h}(T'))). \inlineQED \]
\end{theorem}

Agol also showed that for any measured train track $T$ (representing $\calL$), if we repeated perform maximal splittings then eventually every branch of $T$ split \cite[Lemma~2.1]{A}. For train tracks on the axis of $\calL$ we can give an explicit upper bound on the number of maximal splittings needed.

\begin{lemma}
\label{lem:axis_split_every_branch}
If $\mathbf{T}_{i}$ is a train track on the axis of $\calL$ then every branch of $\mathbf{T}_{i}$ must be split within $3 \zeta m$ maximal splittings.
\end{lemma}

\begin{proof}
Let $B_k$ be the set of branches of $\mathbf{T}_{i}$ that are split within $k m$ maximal splittings. Note that $B_{k+1} = B_k \cup \widehat{h}(B_k)$. Therefore if $B_k$ does not contain of all branches of $\mathbf{T}_{i}$ and $|B_{k+1}| = |B_k|$ then there would be a branch that is never split, which cannot happen \cite[Lemma~2.1]{A}. Hence either $B_k$ consists of all branches of $\mathbf{T}_{i}$ or $|B_{k+1}| > |B_k|$. As $\mathbf{T}_{i}$ has at most $3 \zeta$ branches, the latter case cannot occur when $k \geq 3 \zeta$. Thus every branch of $\mathbf{T}_{i}$ must be split within $3 \zeta m$ maximal splittings.
\end{proof}

We note that by this result every branch of $\mathbf{T}_{i}$ must become a branch of maximal transverse measure within $3 \zeta m$ maximal splittings.

\begin{theorem}
\label{thrm:preperiodic_bound}
The preperiodic length of $T$ is at most $6 \zeta K d(T, \widehat{h}(T))^2$.
\end{theorem}

\begin{proof}
We claim that if $\mathbf{T}_i$ is a measured train track on $A$ then the preperiodic length of a measured train track $T$ is at most $3 \zeta m d(T, \mathbf{T}_i)$. We prove this by induction on the distance. There is nothing to show in the base case when $d(T, \mathbf{T}_i) = 0$ so suppose that the claim is true whenever $d(T, \mathbf{T}_i) < k$. Now if $d(T, \mathbf{T}_i) = k$ then let $T'$ be such that
\[ d(T', T) = 1 \inlineand d(T', \mathbf{T}_i) = k - 1 \]
as shown in Figure~\ref{fig:getting_to_axis}. Then by induction the preperiodic length of $T'$ is at most $3 \zeta m(k-1)$ and so $\mathbf{T}_j \defeq s^{3 \zeta m(k-1)}(T')$ is on the axis $A$. For ease of notation let $T'' \defeq s^{3 \zeta m(k-1)}(T)$, then as $s$ is non-expansive (Lemma~\ref{lem:non_expansive})
\[ d(\mathbf{T}_j, T'') \leq 1. \]
There are now three possibilities that can occur.

The first possibility is that $d(\mathbf{T}_j, T'') = 0$ and so $T'' = \mathbf{T}_j$.

The second possibility is that $d(\mathbf{T}_j, T'') = 1$ and there is a subset $B$ of the branches of $\mathbf{T}_j$ such that every branch in $B$ has the same transverse measure and splitting along these yields $T''$. Now consider $s^p(T'')$ and $s^p(\mathbf{T}_j)$ where, by Lemma~\ref{lem:axis_split_every_branch}, we choose $p < 3 \zeta m$ such that the branches in $B$ appear in $s^p(\mathbf{T}_j)$ with maximal measure. If there are no other branches of maximal measure then maximally splitting $s^p(\mathbf{T}_j)$ results in $s^p(T'')$. Hence $s^{p}(T'') = \mathbf{T}_{j+p+1}$. Otherwise maximally splitting $s^p(\mathbf{T}_j)$ factors through $s^p(T'')$ and so $s^{p+1}(T'') = \mathbf{T}_{j+p+1}$.

The third possibility is that $d(\mathbf{T}_j, T'') = 1$ and there is a subset $B$ of the branches of $T''$ such that splitting along these yields $\mathbf{T}_j$. In this case let $B'$ be new the branches that are added when the branches in $B$ are split. Again consider $s^p(T'')$ and $s^p(\mathbf{T}_j)$ where, by Lemma~\ref{lem:axis_split_every_branch}, we choose $p < 3 \zeta m$ such that the branches in $B'$ appear in $s^p(\mathbf{T}_j)$ with maximal measure. Now in order for $d(s^{p+1}(\mathbf{T}_j), s^{p+1}(T'')) \leq 1$ we must have already split the branches in $B$. Hence $s^{p+1}(T'')$ is must be either $\mathbf{T}_{j+p+1}$ or $\mathbf{T}_{j+p}$.

In any case, it follows that $s^{3 \zeta mk}(T) = s^{3 \zeta m}(T'')$ is either:
\[ \mathbf{T}_{j+3 \zeta m - 1}, \quad \mathbf{T}_{j+3 \zeta m} \inlineor \mathbf{T}_{j+3 \zeta m + 1} \]
and so is on the axis $A$. Hence the claim is true.

Now by Theorem~\ref{thrm:linear_conjugator_bound} there is a $\mathbf{T}_i$ on the axis $A$ such that
\[ d(T, \mathbf{T}_i) \leq K (d(T, \widehat{h}(T)) + d(\mathbf{T}_i, \widehat{h}(\mathbf{T}_i))) = K (d(T, \widehat{h}(T)) + m). \] 
Therefore by the previous claim and Theorem~\ref{thrm:periodic_bound} the preperiodic length of $T$ is at most
\[ 3 \zeta m K (d(T, \widehat{h}(T)) + m) \leq 6 \zeta K d(T, \widehat{h}(T))^2. \qedhere \]
\end{proof}

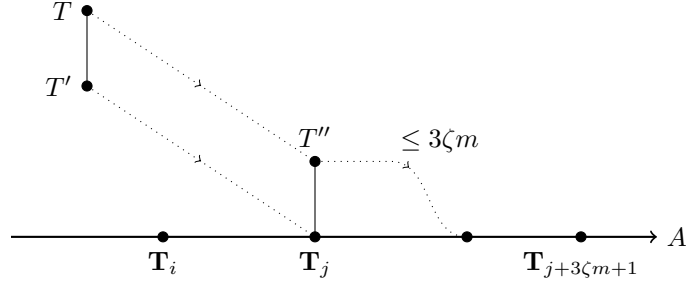
\begin{figure}[ht]
  \centering
  \tikzset{->-/.style={decoration={
  markings,
  mark=at position #1 with {\arrow{>}}},postaction={decorate}}}

\begin{tikzpicture}[scale=1]
\coordinate (T) at (-2,3);
\coordinate (T2) at (-2,2);
\coordinate (T22) at ( 1,1);
\coordinate (Ti) at (-1,0);
\coordinate (Tj) at ( 1,0);
\coordinate (join) at (3, 0);
\coordinate (merged) at (4.5, 0);

\node [draw,shape=circle,fill=black,scale=0.4, label=left:$T$] at (T) {};
\node [draw,shape=circle,fill=black,scale=0.4, label=left:$T'$] at (T2) {};
\node [draw,shape=circle,fill=black,scale=0.4, label=above:$T''$] at (T22) {};
\node [draw,shape=circle,fill=black,scale=0.4, label=below:$\mathbf{T}_i$] at (Ti) {};
\node [draw,shape=circle,fill=black,scale=0.4, label=below:$\mathbf{T}_j$] at (Tj) {};
\node [draw,shape=circle,fill=black,scale=0.4] at (join) {};
\node [draw,shape=circle,fill=black,scale=0.4, label={below:$\mathbf{T}_{j+ 3 \zeta m + 1}$}] at (merged) {};
\draw [->-=0.5, dotted] (T) -- (T22);
\draw (T) -- (T2);
\path [draw, dotted, ->-=0.5] (T22) to (2, 1) node [above, anchor=south west] {$\leq 3\zeta m$} to [out=0, in=180] (join);
\draw [->-=0.5, dotted] (T2) -- (Tj);
\draw (T22) -- (Tj);
\draw [thick, ->] (-3,0) -- (5.5,0) node [right] {$A$};  

\end{tikzpicture}
  \caption{Getting to the axis $A$.}
  \label{fig:getting_to_axis}
\end{figure}

Combining this with Lemma~\ref{lem:filling_axis} we obtain:
\begin{corollary}
\label{cor:filling_lamination}
The measured lamination $\calL$ is filling if and only if $s^t(T)$ is filling, where $t \defeq 6 \zeta K d(T, \widehat{h}(T))^2$. \qed
\end{corollary}

Finally, we note that a triangulation $\calT$ gives rise to a measured train track $T$ as shown in Figure~\ref{fig:tri_to_train_track}. Here the measure assigned to the branch transverse to $e_i$ is $\calL(e_i)$ and the measure on the other branches is determined by the \emph{switch condition} \cite[Page~11]{PH}. Furthermore, if $p$ is a path from $\calT$ to $h(\calT)$ in $G$ then this descends to a quasi-path in $\GG$ and so $d(T, \widehat{h}(T)) \leq 2 \ell(p)$.

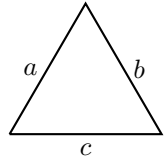
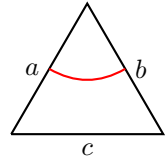
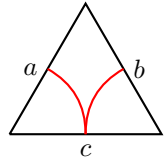
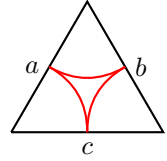
\begin{figure}[ht]
	\centering
	\begin{tabular}{cc}
	\begin{subfigure}[t]{0.4\textwidth}
		\centering
		\begin{tikzpicture}[scale=2,thick]
			\coordinate (A) at (0,0);
			\coordinate (B) at (1,0);
			\coordinate (C) at ($(0.5,{sqrt(3) / 2})$);
			\draw (A) -- node [below] {$c$} (B) -- node [right] {$b$} (C) -- node [left] {$a$} (A);
		\end{tikzpicture}
		\caption{When $\calL(a) = \calL(b) = \calL(c) = 0$.}
	\end{subfigure} &
	\begin{subfigure}[t]{0.4\textwidth}
		\centering
		\begin{tikzpicture}[scale=2,thick]
			\coordinate (A) at (0,0);
			\coordinate (B) at (1,0);
			\coordinate (C) at ($(0.5,{sqrt(3) / 2})$);
			\draw (A) -- node [below] {$c$} (B) -- node [right] {$b$} (C) -- node [left] {$a$} (A);
			\draw [red] ($(A)!0.5!(C)$) to [out=330, in=210] ($(B)!0.5!(C)$);
		\end{tikzpicture}
		\caption{When $\calL(a) = \calL(b)$ and $\calL(c) = 0$.}
	\end{subfigure} \\
	\begin{subfigure}[t]{0.4\textwidth}
		\centering
		\begin{tikzpicture}[scale=2,thick]
			\coordinate (A) at (0,0);
			\coordinate (B) at (1,0);
			\coordinate (C) at ($(0.5,{sqrt(3) / 2})$);
			\draw (A) -- node [below] {$c$} (B) -- node [right] {$b$} (C) -- node [left] {$a$} (A);
			\draw [red] ($(A)!0.5!(C)$) to [out=330, in=90] ($(A)!0.5!(B)$);
			\draw [red] ($(B)!0.5!(C)$) to [out=210, in=90] ($(A)!0.5!(B)$);
		\end{tikzpicture}
		\caption{When $\calL(a) + \calL(b) = \calL(c)$.}
	\end{subfigure} &
	\begin{subfigure}[t]{0.4\textwidth}
		\centering
		\begin{tikzpicture}[scale=2,thick]
			\coordinate (A) at (0,0);
			\coordinate (B) at (1,0);
			\coordinate (C) at ($(0.5,{sqrt(3) / 2})$);
			\draw (A) -- node [below] {$c$} (B) -- node [right] {$b$} (C) -- node [left] {$a$} (A);
			\draw [red] ($(A)!0.5!(C)$) to [out=330, in=90] ($(A)!0.5!(B)$);
			\draw [red] ($(B)!0.5!(C)$) to [out=210, in=90] ($(A)!0.5!(B)$);
			\draw [red] ($(A)!0.5!(C)$) to [out=330, in=210] ($(B)!0.5!(C)$);
		\end{tikzpicture}
		\caption{Otherwise.}
	\end{subfigure}
\end{tabular}
	\caption{A train track coming from a triangulation.}
	\label{fig:tri_to_train_track}
\end{figure}

\section{Certification}
\label{sec:pA}

We can now state the main algorithm for certifying that a mapping class is pseudo-Anosov. Again, we use a path $p$ from $\calT$ to $h(\calT)$ to represent $h$. We use a \emph{certificate} consisting of decimals $x_1, \ldots, x_\zeta$ and polynomials $f_1, \ldots, f_\zeta \in \ZZ[x]$. These decimals represent approximations of the measure assigned to each edge of $\calT$ by the stable lamination of $h$ while these polynomials are their minimal polynomials.

\begin{framed}  
\noindent
Given a path $p$ from $\calT$ to $h(\calT)$ and a certificate $x_1, \ldots, x_\zeta, f_1, \ldots, f_\zeta$, let:
\begin{itemize}
\item $\{A_i\}$ and $\{B_i\}$ be the matrices of Lemma~\ref{lem:encoding_LP},
\item $t \defeq 24 \zeta K \cdot \ell(p)^2$,
\item $h_0 \defeq \zeta^4 (\ell(p) + 6)$,  
\item $h_1 \defeq h_0 + 2 \zeta$,
\item $p_1 \defeq 2 \zeta^2 (2 h_1 + t + 3)$, and
\item $d_1 \defeq p_1 + t + \zeta h_1 + 2$.
\end{itemize}
In each of the following stages, all calculations are done to $d_1$ decimal places and all comparisons are done by \emph{only} comparing the first $p_1$ decimal places.
\begin{enumerate}
\item \label{stp:check_height} Check that each $f_i$ has degree at most $\zeta$ and the log of the absolute value of each coefficient is at most $h_0$.  
\item \label{stp:check_nonneg} Check that each $0 \leq x_i \leq 1$.
\item \label{stp:check_approx} Check that each pair $f_i(x_i \pm 10^{-d_1})$ have different signs.
\item \label{stp:check_triangle} For each face of $\calT$ with edges $a, b, c$ check that $x_a + x_b \geq x_c$.
\item \label{stp:check_peripheral} For each vertex $v \in V(\calT)$ check that there is an incident face of $\calT$ with edges $a, b, c$ such that $v \subseteq a \cap b$ and $x_a + x_b = x_c$.
\item \label{stp:check_unitary} Check that $\sum x_i = 1$.
\item \label{stp:compute_image} Find $B_i$ such that $B_i \cdot (x_1 \; \cdots \; x_\zeta)^T \geq 0$ and compute 
\[ (y_1 \; \cdots \; y_\zeta)^T \defeq A_i \cdot (x_1 \; \cdots \; x_\zeta)^T \inlineand y \defeq \sum y_i. \]
\item \label{stp:check_invariant} Check that each $y_i = y x_i$.
\item \label{stp:check_stable} Check that $y > 1$.
\item \label{stp:check_filling} Check that $s^{t}(T')$ is filling where $T'$ is the measured train track corresponding to $x_1, \ldots, x_\zeta$ (see Section~\ref{sec:train_tracks}).
\end{enumerate}
\end{framed}

We say that a certificate is \emph{accepted} by the main algorithm if every check passes and is \emph{rejected} otherwise. We will show that: 
\begin{itemize}
\item for each $h \in X^*$, there is a certificate that the main algorithm accepts if and only if $h$ is pseudo-Anosov. (Theorem~\ref{thrm:algorithm_correct}), and
\item the main algorithm will accept or reject a certificate of $h$ in $O(\ell(p)^4)$ time (Theorem~\ref{thrm:algorithm_analysis}).
\end{itemize}
As we may choose $p$ such that $\ell(p) \in O(\ell(h))$, these two results suffice to show that the pseudo-Anosov problem is in $\NP$.

\subsection{Algebraic numbers}

In order to prove the correctness of the main algorithm we first recall some properties of algebraic numbers.

\begin{definition}
If $f(x) = \sum a_i x^i \in \ZZ[x]$ is a polynomial then we define its \emph{height} to be 
\[ \height(f) \defeq \log(\max(|a_i|)). \] 
\end{definition}

\begin{definition}
If $\alpha$ is an algebraic number then we define its \emph{height} to be $\height(\alpha) \defeq \height(\mu_\alpha)$ where $\mu_\alpha \in \ZZ[x]$ is its minimal integral polynomial.
\end{definition}

\begin{fact}
\label{fct:heights}
If $\alpha, \beta \in \QQbar$ are algebraic numbers then:
\begin{itemize}
\item $\logdegree(\alpha \pm \beta) \leq \logdegree(\alpha) + \logdegree(\beta)$, where $\logdegree(\alpha) \defeq \log(\deg(\alpha))$,
\item $\height(\alpha \pm \beta) \leq \height(\alpha) + \height(\beta) + 1$ \cite[Property~3.3]{W},
\item $\height(\alpha \beta) \leq \height(\alpha) + \height(\beta)$ \cite[Property~3.3]{W},
\item $\height(\alpha^{-1}) = \height(\alpha)$, and  
\item if $\alpha$ is a root of $f \in \ZZ[x]$ then $\height(\alpha) \leq \height(f) + 2 \deg(f)$ \cite[Corollary~10.12]{BasuPollackRoy}.
\end{itemize}
\end{fact}

Most importantly, algebraic numbers of bounded degree and height are bounded away from zero.

\begin{lemma}[{\cite[Lemma~10.3]{Cohen}}]
\label{lem:algebraic_approximations}
If $\alpha$ is an algebraic number then $\alpha = 0$ if and only if the integer part of $\alpha$ is $0$ and at least the first $\height(\alpha) + \logdegree(\alpha)$ decimal places of $\alpha$ are $0$. \qed
\end{lemma}

\begin{remark}
In certain circumstances the inequalities of Fact~\ref{fct:heights} may be strengthened. For example, if $\alpha_1, \ldots, \alpha_k$ are integers then
\[ \height\left(\sum \alpha_i \right) \leq \max(\height(\alpha_i)) + \log(k) \ll \sum \height(\alpha_i) + (k-1). \]
\end{remark}

\subsection{Acceptance implies pseudo-Anosov}
\label{sub:accept_pA}

We can now prove the correctness of the main algorithm. Namely that $h$ is pseudo-Anosov if and only if there is a certificate that the main algorithm accepts. We begin with the converse direction and proceed by showing that:
\begin{itemize}
\item Stages~\ref{stp:check_height}--\ref{stp:check_approx} show that each $x_i$ is close to an algebraic number (Proposition~\ref{prop:approx_algebraic}).
\item Stages~\ref{stp:check_triangle}--\ref{stp:check_peripheral} show that these algebraic numbers correspond to a measured lamination (Proposition~\ref{prop:exists_lamination}).
\item Stage~\ref{stp:check_unitary} shows that this lamination is \emph{unitary (with respect to $\calT$)}, that is, $||\calT(\calL)|| = 1$, (Proposition~\ref{prop:unitary_lamination}).
\item Stage~\ref{stp:check_invariant} shows that this lamination is projectively invariant under $h$ (Proposition~\ref{prop:invariant_lamination}).
\item Stage~\ref{stp:check_stable} shows that this lamination is stable (Proposition~\ref{prop:stable_lamination}).
\item Stage~\ref{stp:check_filling} shows that this lamination is filling (Proposition~\ref{prop:filling_lamination}).
\end{itemize}

\begin{proposition}
\label{prop:approx_algebraic}
If Stages~\ref{stp:check_height}--\ref{stp:check_approx} of the main algorithm complete then each $x_i$ lies within $10^{-d_1}$ of a unique algebraic number $v_i$ of degree at most $\zeta$ and height at most $h_1$.
\end{proposition}

\begin{proof}
By the intermediate value theorem, Stage~\ref{stp:check_approx} shows that $f_i$ must have a root $v_i$ in $[x_i - 10^{-d_1}, x_i + 10^{-d_1}]$. It follows from Fact~\ref{fct:heights} that $v_i$ is an algebraic number of degree at most $\deg(f_i) \leq \zeta$ and height at most $\height(f_i) + 2 \zeta \leq h_0 + 2 \zeta = h_1$. Finally, any two distinct algebraic number of degree at most $\zeta$ and height at most $h_1$ must be separated by at least $10^{-(2 h_1 + 2 \zeta + 1)} > 10^{-(d_1 - 1)}$ and so $v_i$ is unique.
\end{proof}

From now on fix $v_1, \ldots, v_\zeta$ to be these algebraic numbers.

\begin{proposition}
\label{prop:exists_lamination}
Stages~\ref{stp:check_triangle}--\ref{stp:check_peripheral} of the main algorithm complete if and only if $v_1, \ldots, v_\zeta$ corresponds to a measured lamination $\calL \in \ML(S)$, that is, 
\[ \calT(\calL) = (v_1, \ldots, v_\zeta). \]
\end{proposition}

\begin{proof}
If $x_a + x_b > x_c$ to $p_1$ decimal places then 
\[ v_a + v_b - v_c > 10^{-p_1} - 10^{-d_1} - 10^{-d_1} \geq 10^{-(p_1 - 1)}. \]
However
\[ \height(v_a + v_b - v_c) + \logdegree(v_a + v_b - v_c) \leq 3 h_1 + 3 \zeta \leq p_1 - 1 \]
and so by Lemma~\ref{lem:algebraic_approximations} we have that $v_a + v_b > v_c$. By the same argument if $x_a + x_b = x_c$ to $p_1$ decimal places then $v_a + v_b = v_c$. Hence $v_1, \ldots, v_\zeta$ corresponds to a measured lamination.

Conversely, suppose that $v_1, \ldots, v_\zeta$ corresponds to a measured lamination. If $v_a + v_b > v_c$ then $v_a + v_b - v_c > 10^{-(3 h_1 + 3 \zeta)}$ and so 
\[ x_a + x_b - x_c > 10^{-(3 h_1 + 3 \zeta)} - 10^{-(d_1 - 1)} > 10^{-p_1}. \]
Hence $x_a + x_b > x_c$ to at least $p_1$ decimal places. Similarly if $v_a + v_b = v_c$ then $x_a + x_b = x_c$ to at least $p_1$ decimal places and so Stages~\ref{stp:check_triangle}--\ref{stp:check_peripheral} of the main algorithm will complete.
\end{proof}

From now on fix $\calL$ to be this measured lamination.

\begin{proposition}
\label{prop:unitary_lamination}
Stage~\ref{stp:check_unitary} of the main algorithm completes if and only if $\calL$ is unitary (with respect to $\calT$).
\end{proposition}

\begin{proof}
If $\sum x_i = 1$ to $p_1$ decimal places then $\sum v_i - 1 = 0$ to at least $p_1 - 1$ decimal places. However 
\[ \height\left(\sum v_i - 1\right) + \logdegree\left(\sum v_i - 1\right) \leq \zeta h_1 + \zeta^2 \leq p_1 - 1 \]
and so again by Lemma~\ref{lem:algebraic_approximations} we have that $\sum v_i = 1$. Hence $\calL$ is unitary.

Conversely, if $\calL$ is unitary then $\sum v_i - 1 = 0$ and so $\left|\sum x_i - 1\right| \leq \zeta 10^{-d_1}$. Hence $\sum x_i = 1$ to at least $p_1$ decimal places and so Stage~\ref{stp:check_unitary} of the main algorithm will complete.
\end{proof}

Now following Stage~\ref{stp:compute_image}, let $B_i$ be such that $B_i \cdot (x_1 \; \cdots \; x_\zeta)^T \geq 0$ and fix
\[ (y_1 \; \cdots \; y_\zeta)^T \defeq A_i \cdot (x_1 \; \cdots \; x_\zeta)^T \inlineand y \defeq \sum y_i. \]
As $v_i$ and $x_i$ are so close and comparisons are only done to $p_1$ decimal places we also have that $B_i \cdot (v_1 \; \cdots \; v_\zeta)^T \geq 0$. Hence we fix 
\[ (w_1 \; \cdots \; w_\zeta)^T \defeq A_i \cdot (v_1 \; \cdots \; v_\zeta)^T \inlineand \lambda \defeq \sum w_i. \]
Each $A_i$ is $\ell(p)$--bounded and $\ell(p) \leq h_1$. Therefore each $w_i$ has height at most $h_2 \defeq \zeta (2 h_1 + 1)$ and agrees with $y_i$ to at least $d_2 \defeq d_1 - \zeta h_1$ decimal places. Similarly, $\lambda$ has height at most $h_3 \defeq \zeta (h_2 + 1)$ and agrees with $y$ to at least $d_3 \defeq d_2 - \zeta$ decimal places.

\begin{proposition}
\label{prop:invariant_lamination}
Stage~\ref{stp:check_invariant} of the main algorithm completes if and only if $\calL$ is projectively invariant under $h$, that is, $h(\calL) = \lambda \cdot \calL$.
\end{proposition}

\begin{proof}
Firstly note that $\lambda, y < 10^{h_3 + \zeta}$ and so $y x_i$ and $\lambda v_i$ agree to at least $d_3 - \zeta - h_3$ decimal places. Therefore, if $y_i - y x_i = 0$ to $p_1$ decimal places then $w_i - \lambda v_i = 0$ to at least $p_1 - 1$ places. However
\[ \height(w_i - \lambda v_i) + \logdegree(w_i - \lambda v_i) \leq h_1 + h_2 + h_3 + 2 \zeta^2 \leq p_1 - 1 \]
and so again by Lemma~\ref{lem:algebraic_approximations} we have that $w_i = \lambda v_i$. Hence $\calL$ is projectively invariant under $h$.

Conversely, if $\calL$ is projectively invariant under $h$ then $w_i - \lambda v_i = 0$ and so 
\[ |y_i - y x_i| \leq 10^{-d_2} \leq 10^{-p_1} + 10^{-(d_3 - \zeta - h_3)}. \]
Hence $y_i = y x_i$ to at least $p_1$ decimal places and so Stage~\ref{stp:check_invariant} of the main algorithm will complete.
\end{proof}

Note that completing this stage shows that $(v_1 \; \cdots \; v_\zeta)^T$ is an eigenvector of $A_i$. Hence each $v_i$ lies in $\QQ(\lambda)$ and so any linear combination of them is also an algebraic number of degree at most $\zeta$.

\begin{proposition}
\label{prop:stable_lamination}
Stage~\ref{stp:check_stable} of the main algorithm completes if and only if $\calL$ is stable, that is, $\lambda > 1$.
\end{proposition}

\begin{proof}
If $y - 1 > 0$ to $p_1$ decimal places then $\lambda - 1 > 0$ to at least $p_1 - 1$ decimal places. However
\[ \height(\lambda - 1) + \logdegree(\lambda - 1) \leq h_3 + \zeta \leq p_1 - 1. \]
and so again by Lemma~\ref{lem:algebraic_approximations} we have that $\lambda > 1$. Hence $\calL$ is stable.

Conversely if $\calL$ is stable then $\lambda > 1 + 10^{-(h_3 + \zeta)}$. As $\lambda$ and $y$ agree to at least $d_3$ decimal places we have that 
\[ y > 1 + 10^{-(h_3 + \zeta)} - 10^{-d_3} > 1 + 10^{-p_1}. \]
Hence $y > 1$ to at least $p_1$ decimal places and so Stage~\ref{stp:check_stable} completes.
\end{proof}

Following Section~\ref{sec:train_tracks}, let $T$ be the measured train track obtained from $\calT$ using $v_1, \ldots, v_\zeta$ and let $T_i \defeq s^i(T) = (\tau_i, \mu_i)$. Similarly, let $T'$ be the measured train track obtained from $\calT$ using $x_1, \ldots, x_\zeta$ instead and let $T'_i \defeq s^i(T') = (\tau'_i, \mu'_i)$. We let $v_i^k$ denote the weights on the branches of $T_k$ and $x_i^k$ denote the weights on the branches of $T'_k$.

\begin{proposition}
\label{prop:filling_lamination}
Stage~\ref{stp:check_filling} of the main algorithm completes if and only if $\calL$ is filling.
\end{proposition}

\begin{proof} We will show that $T'_t$ is filling if and only if $T_t$ is. The result then follows directly from Corollary~\ref{cor:filling_lamination}.

We begin by claiming that for $1 \leq k \leq t$:
\[ \tau_k = \tau'_k \inlineand |v_i^k - x_i^k| \leq 10^{-(d_1 - k - 1)}. \]
To see this first note that $v_i^k$ is an algebraic number of degree at most $\zeta$ and $\height(v_i^k) 
\leq 3 \zeta (k + 2 h_1 + 2)$ and that $|v_i^0 - x_i^0| \leq 10^{-(d_1 - 1)}$. Now suppose that $\tau_k = \tau'_k$ and $|v_i^k - x_i^k| \leq 10^{-(d_1 - k - 1)}$ for some $1 \leq k < t$. Then by Lemma~\ref{lem:algebraic_approximations}, $v_i^k \geq v_j^k$ if and only if $x_i^k \geq x_j^k$ to $p_1$ decimal places. Therefore the $x_i^k$--maximal branches are the $v_i^k$--maximal branches and so $\tau_{k+1} = \tau'_{k+1}$. Furthermore 
\[ |v_i^{k+1} - x_i^{k+1}| \leq |v_i^k - x_i^k| + |v_j^k - x_j^k| \leq 10^{-(d_1 - k - 2)} \]
and so the claim holds by induction on $k$.

Finally, again by Lemma~\ref{lem:algebraic_approximations}, we have that $v_i^t > 0$ if and only if $x_i^t > 0$ to $p_1$ decimal places. Hence $T_t$ is filling if and only if $T'_t$ is and so Stage~\ref{stp:check_filling} completes if and only if $\calL$ is filling.
\end{proof}

Combining these propositions we obtain:

\begin{corollary}
\label{cor:pA_if_accepted}
Suppose that $h \in \Mod(S)$ is a mapping class and $p$ is a path from $\calT$ to $h(\calT)$. If there is a certificate that the main algorithm accepts then $h$ is pseudo-Anosov. \qed
\end{corollary}

\subsection{Pseudo-Anosovs have acceptable certificates}

Finally, we show the converse to Corollary~\ref{cor:pA_if_accepted}. To do this we first require some additional bounds on the heights of certain algebraic numbers.

\begin{definition}
Suppose that $\alpha \in \QQbar$ is an algebraic number. A matrix $M$ is \emph{$\alpha$--shifted} if its entries are of the form $a_{ij} = b_{ij} + c_{ij} \alpha$, where $b_{ij}$ and $c_{ij}$ are integers. We say that such a matrix is \emph{$k$--bounded} if each $b_{ij}$ and $c_{ij}$ is.
\end{definition}

\begin{proposition}
\label{prop:height_det}
If $M$ is a $k$--bounded, $m \times m$, $\alpha$--shifted matrix then
\[ \height(\det(M)) \leq m^2 (k + \log(m) + \height(\alpha)). \]
\end{proposition}

\begin{proof}
First note that we may expand $(b_1 + c_1 \alpha) \; \cdots \; (b_m + c_m \alpha)$ as a polynomial in $\alpha$ to obtain $d_0 + \cdots + d_m \alpha^m$. It then follows from Fact~\ref{fct:heights} that $\height(d_i) \leq m k$ as $b_i$ and $c_i$ are $k$--bounded integers .

Now consider the following expansion of $\det(M)$:
\[ \det(M) = \sum_{\sigma \in \Sym(m)} \sgn(\sigma) \prod_{i=1}^m \left(b_{i \sigma(i)} + c_{i \sigma(i)} \alpha\right) = e_0 + \cdots + e_n \alpha^m. \]
By applying the previous bound to the coefficients of $\prod_{i=1}^m \left(b_{i \sigma(i)} + c_{i \sigma(i)} \alpha\right)$ we have that
\[ \height(e_i) \leq m k + m \log(m). \]
Therefore:
\begin{eqnarray*}
\height(\det(M)) &\leq& \height\left(\sum e_i \alpha^i\right) \\
 &\leq& \sum \height(e_i \alpha^i) + m \log(2) \\
 &\leq& \sum \height(e_i) + \frac{1}{2} m^2 \height(\alpha) + m \log(2) \\
 &\leq& m^2 k + m^2 \log(m) + \frac{1}{2} m^2 \height(\alpha) + m \log(2) \\
 &\leq& m^2 (k + \log(m) + \height(\alpha))
\end{eqnarray*}
Thus showing the required bound.
\end{proof}

\begin{lemma}
\label{lem:height_eigenvector}
Suppose that $M$ is a $k$--bounded, $m \times m$, $\alpha$--shifted matrix and that $\det(M) \neq 0$. If $v = (\alpha_1 \; \cdots \; \alpha_m)^T$ is a vector of algebraic numbers such that $M \cdot v$ is a $k$--bounded vector of integers then
\[ \height(\alpha_i) \leq 2 m^2 (k + \log(m) + \height(\alpha)) \]
\end{lemma}

\begin{proof}
This bound follows from using Cramer's rule to determine $\alpha_i$ and Proposition~\ref{prop:height_det}.
\end{proof}

\begin{lemma}
\label{lem:height_eigenvalue}
Suppose that $M$ is a $k$--bounded $m \times m$ integer matrix. If $\alpha$ is an eigenvalue of $M$ then $\height(\alpha) \leq m k + m \log(m) + 2m$.
\end{lemma}

\begin{proof}
Again, following the proof of Proposition~\ref{prop:height_det}, we have that if $\chi_M(x) = e_0 + \cdots + e_n x^m$ is the characteristic polynomial of $M$ then $\height(e_i) \leq m k + m \log(m)$. As $\alpha$ is a root of this it follows from Fact~\ref{fct:heights} that
\[ \height(\alpha) \leq m k + m \log(m) + 2m. \qedhere \]
\end{proof}

We now have the tools to prove existence of a certificate that the main algorithm will accept for pseudo-Anosovs.

\begin{theorem}
\label{thrm:pA_have_certifiacte}
Suppose that $h \in \Mod(S)$ is a mapping class and $p$ is a path from $\calT$ to $h(\calT)$. If $h$ is pseudo-Anosov then there is a certificate that the main algorithm will accept.
\end{theorem}

\begin{proof}
Let $\calL = \calL^+(h)$ be the stable lamination of $h$ scaled such that $||\calT(\calL)|| = 1$. Let $v_i \defeq \calL(e_i)$, where $e_i$ are the edges of $\calT$, and $\mathbf{v} \defeq (v_1 \; \cdots \; v_\zeta)^T$. Let $\{A_i\}$ and $\{B_i\}$ be the matrices of Lemma~\ref{lem:encoding_LP}. There is $B_i$ such that $B_i \cdot \mathbf{v} \geq 0$ and so $\mathbf{v}$ is an eigenvector of $A_i$. Therefore, as $A_i$ is $\ell(p)$--bounded, each $v_i$ is an algebraic number of degree at most $\zeta$ and by Lemma~\ref{lem:height_eigenvector} and Lemma~\ref{lem:height_eigenvalue} we have that
\begin{eqnarray*}
\height(v_i) &\leq& 2 \zeta^2 (\ell(p) + \log(\zeta) + \zeta \ell(p) + \zeta \log(\zeta) + 2 \zeta) \\
 &\leq& 2 \zeta^2 ((\zeta + 1) \ell(p) + (\zeta + 1) \log(\zeta) + 2 \zeta) \\
 &\leq& \zeta^4 (\ell(p) + 6) \\
 &=& h_0.
\end{eqnarray*}

Now let $x_i$ be a decimal approximation of $v_i$, correct to $d_1$ decimal places, and let $f_i$ be the minimal integral polynomial of $v_i$. We claim that the certificate
\[ x_1, \ldots, x_\zeta, f_1, \ldots, f_\zeta \]
is accepted by the main algorithm. To see this consider the first three stages of the main algorithm:
\begin{enumerate}
\item By definition each $f_i$ has degree at most $\zeta$ and height at most $h_0$. Hence this stage will pass.
\item As $0 \leq v_i \leq 1$ and $|x_i - v_i| \leq 10^{-d_1}$, we have that $0 \leq x_i \leq 1$ to at least $p_1$ decimal places. Hence this stage will pass.
\item By Fact~\ref{fct:heights}, two distinct roots of $f_i$ must be separated by at least $10^{-(2 h_1 + \zeta)}$ and as $f_i$ is minimal it has no repeated roots. Hence, $v_i$ is the unique root of $f_i$ in $[x_i - 10^{-d_1}, x_i + 10^{-d_1}]$ and so $f_i(x_i \pm 10^{-d_1})$ must have different signs. Hence this stage will pass and the algebraic numbers found will be $v_1, \ldots, v_\zeta$.
\end{enumerate}
Finally, note that as $v_1, \ldots, v_\zeta$ represents $\calL$ which is unitary (with respect to $\calT$), projectively invariant, stable and filling Stages~\ref{stp:check_unitary}--\ref{stp:check_filling} must complete by Proposition~\ref{prop:unitary_lamination}, Proposition~\ref{prop:invariant_lamination}, Proposition~\ref{prop:stable_lamination} and Proposition~\ref{prop:filling_lamination}. Proving the claim that the main algorithm will accept this certificate.
\end{proof}

Together with Corollary~\ref{cor:pA_if_accepted} this shows:

\begin{theorem}
\label{thrm:algorithm_correct}
Suppose that $h \in \Mod(S)$ is a mapping class and $p$ is a path from $\calT$ to $h(\calT)$. There is a certificate that the main algorithm will accept if and only if $h$ is pseudo-Anosov. \qed
\end{theorem}

\subsection{Analysis}
\label{sub:analysis}

Finally, we consider the running time of the main algorithm.

\begin{theorem}
\label{thrm:algorithm_analysis}
The main algorithm will accept or reject a certificate of $h$ in $O(\ell(p)^4)$ time.
\end{theorem}

\begin{proof}
We analyse each of the stages of the main algorithm in turn. Recall that $h_0, h_1 \in O(\ell(p))$ and $t, d_1, p \in O(\ell(p)^2)$.
\begin{enumerate}
\item This can be done in $O(h_0) = O(\ell(p))$ operations.
\item This can be done in $O(p_1) = O(\ell(p)^2)$ operations.
\item By expanding $f_i(x) = \sum a_j x^j$ as $a_0 + x (\cdots (a_{\zeta-2} + x(a_{\zeta-1} + a_\zeta x)) \cdots )$ using Horner's rule \cite[Section~4.6.4]{AoCPII} it can be seen that each $f_i(x_i \pm 10^{-d_1})$ can be computed in $O(d_1(d_1 + \zeta h_0)) = O(\ell(p)^4)$. Hence this stage can be done in $O(\ell(p)^4)$ operations.
\item There are $O(1)$ inequalities to check each of which requires $O(d_1 + p_1) = O(\ell(p)^2)$ operations. Hence this can be done in $O(\ell(p)^2)$ operations.
\item By the same argument this also requires $O(\ell(p)^2)$ operations.
\item This can be done in $O(d_1 + p_1) = O(\ell(p)^2)$ operations.
\item Following the ideas of Remark~\ref{rem:compute_image}, as each $x_i$ has $d_1$ digits, we can find $A_i$ and $B_i$ and compute $y_1, \ldots, y_\zeta$ in $d_1 \ell(p) + \ell(p)^2 \in O(\ell(p)^3)$ operations. Each $y_i$ has at most $d_1 + \ell(p) = O(\ell(p)^2)$ digits. Hence, $y$ can be computed in $O(\ell(p)^2)$ further operations and has $O(\ell(p)^2)$ digits. Therefore this entire stage can be done in $O(\ell(p)^3)$ operations.
\item As $x_i$, $y_i$ and $y$ each have $O(\ell(p)^2)$ digits this stage can be done in $O(\ell(p)^4 + d_1 + p_1) = O(\ell(p)^4)$ operations.
\item As $y$ has $O(\ell(p)^2)$ digits, this can be done in $O(\ell(p)^2 + p_1) = O(\ell(p)^2)$ operations.
\item Constructing $T'$ takes at most $O(d_1) = O(\ell(p)^2)$ operations. For each maximal splitting it takes $O(p_1)$ operations to find the maximal weight branches and then $O(d_1)$ further operations to perform the splitting. Therefore we can construct $T'_t \defeq s^{t}(T')$ in $O(t (p_1 + d_1)) = O(\ell(p)^4)$ operations. Finally, checking whether $T'_t$ is filling can be done in $O(p_1) = O(\ell(p)^2)$ operations. Hence this whole stage can be done in $O(\ell(p)^4)$ time.
\end{enumerate}
Therefore the main algorithm will terminate in $O(\ell(p)^4)$ time.
\end{proof}

Again, we may choose $p$ such that $\ell(p) \in O(\ell(h))$. Hence we deduce that:

\begin{corollary}
\label{cor:pA_NP}
Fix $S$, a marked surface, and $X$, a finite generating set of $\Mod(S)$. Deciding whether a word $h \in X^*$ correspond to a pseudo-Anosov mapping classes is a problem in $\NP$. \qed
\end{corollary}

As usual, this immediately gives an exponential time algorithm to determine whether a word is pseudo-Anosov.

\section{The conjugacy problem}
\label{sec:conjugacy}

In this section we apply the results developed in Section~\ref{sec:train_tracks} to the conjugacy problem for mapping class groups. We use these to show that that the conjugacy problem is in $\coNP$ (Corollary~\ref{cor:conjugacy_coNP}). That is, if two words do not represent conjugate mapping classes then there is a short proof of this.

To do this we work with a path $p$ from $\calT$ to $h(\calT)$ and deal with three cases depending on the Nielsen--Thurston type of $h$. In each case we will describe a total conjugacy invariant and show that this can be constructed in $O(\poly(\ell(p)))$ time when given a suitable certificate. Thus, if $q$ is a second path from $\calT'$ to $g(\calT')$ then these invariants are small enough they can be compared in $O(\poly(\ell(p) + \ell(q)))$ operations. Hence we can determine whether $g$ and $h$ are conjugate in polynomial time. 

Again we may choose paths $p$ and $q$ such that $\ell(p) \in O(\ell(h))$ and $\ell(q) \in O(\ell(g))$. Hence this shows that there are certificates that allow us to determine whether $g$ and $h$ are conjugate in $O(\poly(\ell(g) + \ell(h)))$ operations.

The reducible case relies on a solution to a stronger version of the conjugacy problem for the other two cases. Specifically, a solution to the \emph{permutation conjugation problem}:

\begin{problem}[The permutation conjugacy problem]
\label{prob:conjugating_permutation}
Suppose that $g, h \in \Mod(S)$ are mapping classes. Given paths $p$ from $\calT$ to $h(\calT)$ and $q$ from $\calT$ to $g(\calT')$ and a map $\pi \from V \subseteq V(\calT) \to V(\calT')$, decide whether $g$ and $h$ are \emph{$\pi$--conjugate}. That is, decide whether there is there is a mapping class $f \in \Mod(S)$ such that $f g f^{-1} \equiv h$ and $f|_V = \pi$.
\end{problem}

\subsection{Periodic mapping classes}
\label{sub:conjugacy_periodic}

We begin with the case in which $h$ is periodic. Here we use properties from its quotient orbifold as a total conjugacy invariant. 

As $h$ is periodic, by the Nielsen realization theorem \cite[Theorem~7.2]{FM} there is a homeomorphism $\phi \in h$ such that $\ord(\phi) = \ord(h)$. We may use this to define the \emph{quotient orbifold} $\calO \defeq \quotient{S}{\phi}$ which, up to homeomorphism, is independent of the particular choice of $\phi$. There are two key properties that we will need to determine $\calO$.

Firstly, we say that the \emph{order} of a point $x \in S$ is
\[ \ord(x) \defeq \begin{cases}
\min\{k > 0 : \phi^k(x) = x\} & \textrm{if $x$ is not a marked point,} \\
-\min\{k > 0 : \phi^k(x) = x\} & \textrm{otherwise.}
\end{cases} \]
We note that all but finitely many of the points of $S$ are \emph{regular}, that is, have order $\ord(\phi)$. In fact the number of irregular points is bounded only in terms of the topology of $S$.

Secondly, even though $\phi^{|\ord(x)|}(x) = x$ this map may not act like the identity near $x$; it may rotate the tangent plane at $x$ by $2 \pi r(x)$. We refer to the rational number $r(x)$ as the \emph{rotation number} of $x$. We note that the denominator of $r(x)$ is bounded above by $\ord(\phi) \leq 8 \genus(S) + 4 \nummarkedpoints(S) - 2$ \cite[Theorem~7.5]{FM}.

Let $N(o, r)$ denote the number of irregular points of $S$ with order $o$ and rotation number $r$. Nielsen showed that $N$ determines $\calO$ and so is a total conjugacy invariant of $h$ \cite[Theorem~9]{MosherLecture}.

To compute $N$ we consider a \emph{multiarc}, that is, is the isotopy class of the image of a smooth proper embedding of a finite number of copies of $[0, 1]$ (whose endpoints are sent to marked points) into $S$. Again we represent a multiarc via its intersection numbers with the edges of a triangulation. However, as a non-trivial multiarc can have zero intersection with all edges, we must first make a slight modification to the standard definition of intersection number:

\begin{definition}
If $\alpha$ is a multiarc which contains $k$ copies of the edge $e$ of $\calT$ then their \emph{intersection number} is defined to be $\intersection(\alpha, e) \defeq -k$.
\end{definition}

Having made this change we can now identify a multiarc with its \emph{normal coordinate}, its vector of intersection numbers with the edges of $\calT$. These coordinates allow us to restate many of the results of Section~\ref{sub:model_ML} for multiarcs.

Firstly, as in Lemma~\ref{lem:is_multicurve_LP}, there are $O(1)$--bounded $\zeta \times 3\zeta$ matrices $F'_i$ such that a vector $v$ corresponds to a multiarc if and only if $v \neq 0$ and
\[ F'_i \cdot v \geq_2 0 \]
for some $i$. Secondly, analysis of the $30$ cases that can occur within a pair of triangles shows how these coordinates change under an edge flip.

\begin{lemma}
Suppose that $\alpha$ is a multiarc and $e$ is a flippable edge of a triangulation $\calT$ as shown in Figure~\ref{fig:flip} then
{
\newcommand{\waa}{\widehat{a}}
\newcommand{\wbb}{\widehat{b}}
\newcommand{\wcc}{\widehat{c}}
\newcommand{\wdd}{\widehat{d}}
\newcommand{\wee}{\widehat{e}}
\newcommand{\eee}{\bar{e}}
\[ \intersection(\alpha, e') = \begin{cases}
\waa + \wbb - \eee & \textrm{if }
\eee \geq \waa + \wbb \textrm{ and } \waa \geq \wdd \textrm{ and } \wbb \geq \wcc, \\
\wcc + \wdd - \eee & \textrm{if }
\eee \geq \wcc + \wdd \textrm{ and } \wdd \geq \waa \textrm{ and } \wcc \geq \wbb, \\

\waa + \wdd - \eee & \textrm{if }
\eee \leq 0 \textrm{ and } \waa \geq \wbb \textrm{ and } \wdd \geq \wcc, \\
\wbb + \wcc - \eee & \textrm{if }
\eee \leq 0 \textrm{ and } \wbb \geq \waa \textrm{ and } \wcc \geq \wdd, \\

\waa + \wdd - 2 \eee & \textrm{if }
\eee \geq 0 \textrm{ and } \waa \geq \wbb + \eee \textrm{ and } \wdd \geq \wcc + \eee, \\
\wbb + \wcc - 2 \eee & \textrm{if }
\eee \geq 0 \textrm{ and } \wbb \geq \waa + \eee \textrm{ and } \wcc \geq \wdd + \eee, \\

\frac{1}{2}(\waa + \wbb - \eee) & \textrm{if }
\waa + \wbb \geq \eee \textrm{ and } \wbb + \eee \geq 2 \wcc + \waa \textrm{ and } \waa + \eee \geq 2 \wdd + \wbb, \\
\frac{1}{2}(\wcc + \wdd - \eee) & \textrm{if }
\wcc + \wdd \geq \eee \textrm{ and } \wdd + \eee \geq 2 \waa + \wcc \textrm{ and } \wcc + \eee \geq 2 \wbb + \wdd, \\

\max(\waa + \wcc, \wbb + \wdd) - \eee & \textrm{otherwise}
\end{cases} \]
where $\widehat{x} \defeq \max(\intersection(\alpha, x), 0)$ and $\eee \defeq \intersection(\alpha, e)$. \qed
}
\end{lemma}

Now suppose that $\alpha$ is a multiarc in $\calO$ which \emph{fills}, that is, such that each component of $\calO - \alpha$ is a disk containing at most one irregular point in its interior. For example, see Figure~\ref{fig:orbifold_arc}. By lifting $\alpha$ back to $S$ we see that there is an $h$--invariant filling multiarc $\beta$ on $S$. Now $\beta$ decomposes $S$ into polygonal pieces, each of which contains at most one irregular point. For example, see Figure~\ref{fig:lift_of_arc}. These polygons are permuted by the action of $h$ and it is from this permutation that we can immediately determine the orders and rotation numbers of the irregular points of $S$ and so compute $N$.

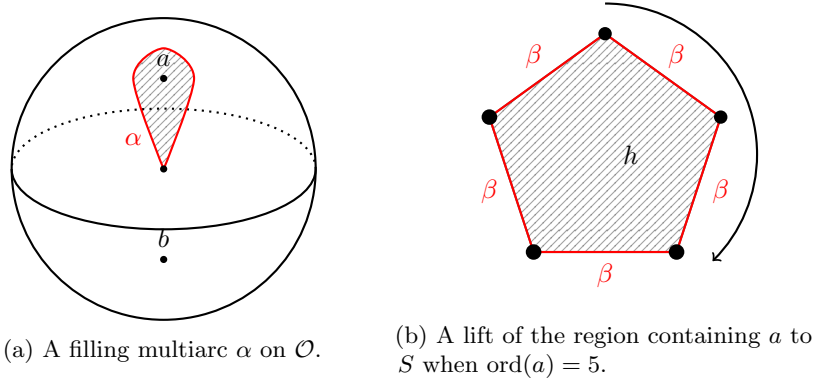
\begin{figure}[ht]
    \centering
    \begin{subfigure}[c]{0.45\textwidth}
      \centering
      \begin{tikzpicture}[thick, scale=0.8]

\tikzset{dot/.style={draw,shape=circle,fill=black,scale=0.2}}

\draw (0,0) circle (2.5);
\draw (-2.5,0) arc (180:360:2.5 and 1);
\draw [dotted] (2.5,0) arc (0:180:2.5 and 1);

\fill [pattern=north east lines,opacity=0.6] plot [smooth] coordinates {(0,0) (-0.5,1.5) (0,2) (0.5,1.5) (0,0)};
\draw [red] plot [smooth] coordinates {(0,0) (-0.5,1.5) (0,2) (0.5,1.5) (0,0)};
\node [red] at (-0.5,0.5) {$\alpha$};

\node [dot] at (0,1.5) [label={\contour*{white}{$a$}}] {};
\node [dot] at (0,-1.5) [label={$b$}] {};
\node [dot] at (0,0) {};

\end{tikzpicture}
      \caption{A filling multiarc $\alpha$ on $\calO$.}
      \label{fig:orbifold_arc}
    \end{subfigure}
    ~ 
    \begin{subfigure}[c]{0.45\textwidth}
      \centering
      \begin{tikzpicture}[thick,scale=0.8]

\tikzset{dot/.style={draw,shape=circle,fill=black,scale=0.2}}

\newdimen\R
\R=2cm

\draw [pattern=north east lines,opacity=0.6] (18:\R) \foreach \x in {90,162,...,359} {-- (\x:\R)} -- cycle;
\draw [red] (18:\R) \foreach \x in {90,162,...,359} {-- (\x:\R)} -- cycle;
\foreach \i in {18,90,...,359} \node [dot] at (\i: \R) {$\i$};
\foreach \i in {18,90,...,359} \node at (\i+36: \R) [red] {$\beta$};

\draw [->] (0,2.5) arc (90:-45:2.5cm) node [midway, label=right:{$h$}] {};

\draw [opacity=0] (0,0) circle (2.5);

\end{tikzpicture}
      \caption{A lift of the region containing $a$ to $S$ when $\ord(a)=5$.}
      \label{fig:lift_of_arc}
    \end{subfigure}
    \caption{An invariant multiarc $\beta$ obtained by lifting a multiarc $\alpha$.}
    \label{fig:lift_arc}
\end{figure}

Thus, given a path $p$ from $\calT$ to $h(\calT)$ and vector $v \defeq \calT(\beta)$ we:
\begin{enumerate}
\item Check that $v$ corresponds to a multiarc $\beta$ by checking that $F'_i \cdot v \geq_2 0$ for some $i$.
\item Check that $\beta$ is $h$--invariant by computing and checking that $\calT(h(\beta)) = v$.
\item Compute the normal coordinate of each component $\beta_i$ of $\beta$ and check that each occurs with multiplicity one.
\item Compute the components $\beta_i$ that are adjacent to each region of $S - \beta$, together with their cyclic ordering.
\item Compute the permutation of $\beta_i$ induced by $h$ and so compute $N$.
\end{enumerate}

To analyse the number of operations required by this procedure, suppose that the given vector $v$ is $k$--bounded. Then checking that $F'_i \cdot v \geq_2 0$ for some $i$ can be done in $O(\poly(k))$ operations as the $F'_i$ matrices are $O(1)$--bounded. By repeating the idea of Remark~\ref{rem:compute_image}, we can compute $\calT(h(\beta))$ and so complete Stage~2 in $O(\poly(k) \poly(\ell(p)))$ operations. 

To complete Stage~3 and Stage~4 we can use the algorithm of Agol, Hass and Thurston \cite{AHT} or one of its variants \cite[Section~6]{EricksonNayyeri} \cite[Theorem~1]{SchaeferSedgwick}. These allow us to extract the edge vectors of the components of $\beta$ in $O(\poly(k))$ operations. Finally by computing $h(\beta_i)$ using the idea of Remark~\ref{rem:compute_image}, this permutation can be computed in $O(\poly(k) \poly(\ell(p)))$ operations and from this $N$ can be computed in $O(\poly(k) \poly(\ell(p)))$ operations too.

Now if $p$ is a path from $\calT$ to $h(\calT)$ we can once again compute matrices $A'_i$ and $B'_i$ which describe the piecewise linear transformation between coordinates with respect to $\calT$ and $h(\calT)$. Again, these matrices can be chosen such that:
\begin{itemize}
\item each $A'_i$ and $B'_i$ is $\ell(p)$--bounded,
\item each $B'_i$ has $O(\ell(p))$ rows,
\item for each multiarc $\alpha$ we have that $B'_i \cdot \calT(\alpha) \geq 0$ for some $i$, and
\item for each multicurve $\alpha$ we have that
\[ \calT(h(\alpha)) = A'_i \cdot \calT(\alpha) \quad \textrm{if and only if} \quad B'_i \cdot \calT(\alpha) \geq 0. \]
\end{itemize}
We can use $A'_i$, $B'_i$ and $F'_i$ to repeat the argument of \cite[Theorem~3.3]{BellReducibility}. From this we deduce that there is a filling multiarc $\beta$ on $S$ such that $h(\beta) = \beta$ and $\calT(\beta)$ is $O(\ell(p))$--bounded. Hence there is a $v$ such that we can compute $N$ in $O(\poly(\ell(p)))$ operations.

\begin{corollary}
\label{cor:p_conjugate}
Deciding whether two periodic words $g, h \in X^*$ correspond to conjugate mapping classes is a problem in $\coNP$. \qed
\end{corollary}

To deal with the permutation conjugacy problem we use the following lemma.

\begin{lemma}
Periodic mapping classes $h$ and $g$ are $\pi$--conjugate if and only if:
\begin{itemize}
\item $h$ and $g$ are conjugate, and
\item for each marked point $v \in V(S)$ we have that:
\begin{itemize}
\item $v$ and $\pi(v)$ have the same rotation number (with respect to $h$ and $g$ respectively), and
\item $\pi(h^k(v)) = g^k(\pi(v))$ for every $k \in \ZZ$
\end{itemize}
whenever these maps are defined.
\end{itemize}
\end{lemma}

\begin{proof}
The forward direction of this lemma holds trivially. For the reverse direction, we first note that without loss of generality we may assume that $\pi$ is actually a permutation of the vertices of $S$. If it is not then we consider each of the possible extensions of $\pi$ in turn.

As $h$ and $g$ are conjugate there is a mapping class $f$ such that $h = f^{-1} g f$. If $f|_V \neq \pi$ then consider $\varphi \from S / h \to S / g$, the homeomorphism induced by $f$. This homeomorphism respects the orders and rotation numbers of points and satisfies the \emph{lifting criterion}: it maps the subgroup $\pi_1(S) \leq \pi_1(S / h)$ to the subgroup $\pi_1(S) \leq \pi_1(S / g)$. We can modify $\varphi$ by precomposing it with another homeomorphism $\psi \from S/h \to S/h$. If we take $\psi$ to be a homeomorphism that swaps two of the marked points whose lifts have the same rotation number then $\psi$ preserves the subgroup $\pi_1(S) \leq \pi_1(S / h)$. Hence $\varphi \circ \psi$ also satisfies the lifting criterion and so lifts to an alternate mapping class $f'$ which, like $f$, conjugates $g$ to $h$ but whose action on $V$ is permuted by $\psi$.

Therefore, as $\pi$ sends marked points to marked points with the same rotation number and $\pi(h^k(v)) = g^k(\pi(v))$ for every $k \in \ZZ$, modifications of this form are sufficient to adjust $f$ such that $f|_V = \pi$. Hence $h$ and $g$ are $\pi$--conjugate.
\end{proof}

The rotation numbers of the marked points can be determined in polynomial time from the polygonal decomposition of $S$ given by the multiarc $\beta$. Thus this additional criterion can also be tested in polynomial time.

\subsection{Pseudo-Anosov mapping classes}
\label{sub:aperiodic_irreducible}

Agol showed that for pseudo-Anosov mapping classes the combinatorics of the periodic part of their maximal splitting sequence is a total conjugacy invariant \cite[Section~7]{A}. Thus, as we can construct these in polynomial time when given an acceptable certificate for $p$, we immediately obtain that:

\begin{corollary}
\label{cor:ap_ir_conjugate}
Deciding whether two aperiodic irreducible words $g, h \in X^*$ correspond to conjugate mapping classes is a problem in $\coNP$. \qed
\end{corollary}

We also note that in this case $g$ and $h$ are $\pi$--conjugate if and only if they are conjugate and there is a map between the combinatorics of the periodic part of their maximal splitting sequence which induces $\pi$ on $V$. Again, given acceptable certificates for $p$ and $q$, this can be done in $O(\poly(\ell(p) + \ell(q)))$ operations.

\subsection{Reducible mapping classes}
\label{sub:aperiodic_reducible}

To deal with the final case in which $h$ is reducible (and aperiodic) we use its partition graph as a total conjugacy invariant. This involves decomposing the given mapping class along its canonical curve system and considering the maps induced on the invariant subsurfaces.

\subsubsection{Crushing}
\label{subsub:crushing}

In order to be able to construct the map induced on invariant subsurfaces computationally, we first recall the notion of \emph{crushing} a surface along a curve from \cite[Section~4]{BellReducibility}.

\begin{definition}
We \emph{crush} $S$ along $\gamma$ to obtain the (possibly disconnected) surface $S_\gamma$ by removing an open regular neighbourhood of $\gamma$ and then collapsing the new boundary components to additional marked points. Note that in order for this surface to be triangulable we must also discard any components that are twice marked spheres. See Figure~\ref{fig:crush} for example.
\end{definition}

\begin{figure}[ht]
\centering
\begin{tikzpicture}[scale=2,thick]
\node (a) at (-0.25, 0) [anchor=east] {\includegraphics[width=0.4\linewidth, angle=0]{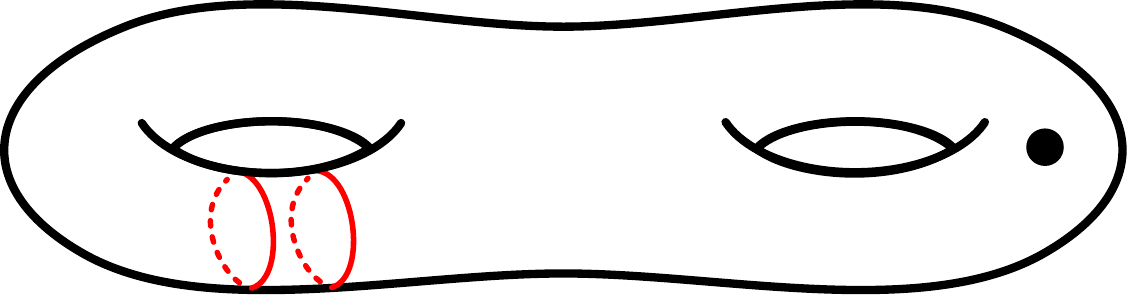}};
\node (b) at ( 0.25, 0) [anchor=west] {\includegraphics[width=0.4\linewidth, angle=0]{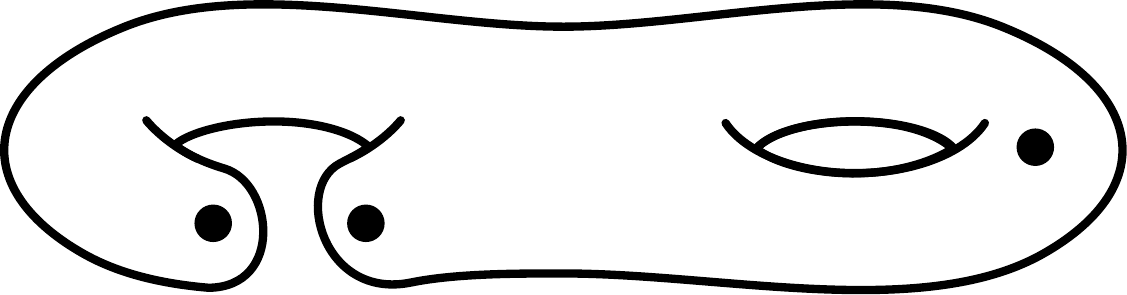}};
\draw [thick,->] ($(a.east)!0.1!(b.west)$) -- node[above] {Crush} ($(a.east)!0.9!(b.west)$);
\end{tikzpicture}
\caption{Crushing along a multicurve.}
\label{fig:crush}
\end{figure}

Now if $\calT$ is a triangulation of $S$ then we may track it as we crush $S$ along a multicurve $\gamma \in \calC(S)$. After collapsing any bigons that are created, this results in a triangulation $\calT_\gamma$ of $S_\gamma$ which also has $\zeta$ edges. Erickson and Nayyeri showed how to construct $\calT_\gamma$ in polynomial time:

\begin{theorem}[{\cite{EricksonNayyeri}}]
\label{thrm:crush_complexity}
There is an algorithm to compute $\calT_\gamma$ which completes in $O(\poly(k))$ operations when $\calT(\gamma)$ is $k$--bounded. \qed
\end{theorem}

\begin{proposition}
\label{prop:crush_paths}
If $p$ is a path from $\calT$ to $\calT'$ then crushing each triangulation of $p$ along $\gamma$, and possibly discarding any repeated triangulations, gives a path $p_\gamma$ from $\calT_\gamma$ to $\calT'_\gamma$.
\end{proposition}

\begin{proof}
The result clearly holds when $p$ consists of a single reordering of the edges of $\calT$. If $p$ consists of a single flip then the combinatorics of $\calT_\gamma$ and $\calT'_\gamma$ agree away from the faces coming from the faces incident to the flipped edge. Thus $\calT_\gamma$ and $\calT'_\gamma$ share at least $\zeta - 1$ edges and so they are either equal or differ by a single flip. The result then follows for all paths by induction on $\ell(p)$.
\end{proof}

In fact when $\calT'$ is obtained by flipping the edge $e$ of $\calT$, we have that $\calT_\gamma$ and $\calT'_\gamma$ are equal if and only if there is an arc of $\gamma$ passing from one side of the square containing $e$ to the opposite side. Following the notation of Figure~\ref{fig:flip}, this occurs if and only if $\intersection(\gamma, a) + \intersection(\gamma, c) \neq \intersection(\gamma, b) + \intersection(\gamma, d)$.

We note that by construction $\ell(p_\gamma) \leq \ell(p)$.

\begin{corollary}
Suppose that $p$ is a path from $\calT$ to $\calT'$ and $\gamma \in \calC(S)$ is a $k$--bounded multicurve. If $\calT(\gamma)$ is $k$--bounded then we can compute $p_\gamma$ in $O(\ell(p) \poly(k))$ operations. \qed
\end{corollary}

\subsubsection{Partition graphs}
\label{subsub:partition_graphs}

Now let $\canonical(h) \neq \emptyset$ denote the \emph{canonical curve system} of $h$ \cite[Page~373]{FM}. This is the intersection of all maximal $h$--invariant multicurves. We will assume that $\canonical(h)$ has components $\{\gamma_j\}$. Let $S_{\canonical(h)}$ be the surface obtained by crushing $S$ along $\canonical(h)$ and for ease of notation let $\{S_i\}$ be its connected components. Additionally, let $h_i$ be the mapping class induced on $S_i$ by the first return map of $h$.

\begin{definition}[{\cite[Theorem~2]{MMa}, \cite{MosherLecture}}]
The \emph{partition graph} of $h$ is the pair $(H, \phi)$ where:
\begin{itemize}
\item $H$ is the finite graph with:
\begin{itemize}
\item a vertex corresponding to each $h_i$, and
\item an edge corresponding to each $\gamma_i$, connecting between $h_j$ and $h_k$ when $\gamma_i$ meets $S_j$ and $S_k$.
\end{itemize}
\item $\phi$ is the automorphism of $H$ induced by $h$.
\end{itemize}
\end{definition}

Most importantly, we can provide a polynomial-time verifiable certificate that a given graph is the partition graph of $h$. This certificate consists of two pieces of information. Firstly $\calT(\canonical(h))$ which is $O(\ell(p))$--bounded \cite[Corollary~4.9]{BellReducibility} \cite[Theorem~1.1]{KoberdaMangahas}. From this we can compute paths representing each $h_i$ in $O(\poly(\ell(p)))$ operations by using the algorithm of Theorem~\ref{thrm:crush_complexity}. Secondly a certificate accepted by the main algorithm for each pseudo-Anosov $h_i$, this allows us to deduce that the given multicurve is $h$--maximal. 

Such information also allows us to verify that the given multicurve is actually the canonical curve system. To do this we note that it is sufficient to check that for any $\gamma_i$, removing its orbit under $h$ from $\canonical(h)$ does not result in a new periodic component. As there are at most $\zeta \in O(1)$ such orbits to check and again we can compute each $\calT(\gamma_i)$ in $O(\poly(\ell(p)))$ operations \cite[Section~6]{EricksonNayyeri} \cite{AHT} \cite[Theorem~1]{SchaeferSedgwick}, this can also be done in polynomial time.

\subsubsection{Twist invariants}
\label{subsub:twist_invariants}

Associated to each component $\gamma_i$ of the canonical curve system is its \emph{twist invariant} $d_i$. This is a rational number describing the number of (fractional) Dehn twists performed about it \cite{MosherLecture}. We now explicitly describe how to compute these numbers given some additional multicurves.

\begin{definition}
A curve $\gamma$ is \emph{dual} to a component $\gamma_i$ of $\canonical(h)$ if
\[ \intersection(\gamma, \gamma_i) = \begin{cases}
2 & \textrm{ if $\gamma_i$ is separating} \\
1 & \textrm{ otherwise}.
\end{cases} \]
\end{definition}

Now suppose that $\gamma$ is dual to $\gamma_i$ and let $\delta \defeq \partial N(\gamma_i \cup \gamma)$. Let $k \defeq \lcm(1, 2, \ldots, 4 \zeta)$ then
\[ |d_i| = \frac{\intersection(h^k(\gamma), \gamma) - \frac{1}{2}\intersection(h^k(\gamma), \delta)}{k \intersection(\gamma, \gamma_i)}. \]
This particular value of $k$ was chosen such that $h^k$ fixes all components of $\canonical(h)$ and all prongs of singularities of any measured lamination that is projectively invariant under $h$.

To compute the sign of $d_i$ we check whether the number of intersections grows when an additional Dehn twist along $\gamma_i$ is performed. That is, $d_i \geq 0$ if and only if
\[ \intersection(h^k(T_{\gamma_i}(\gamma)), \gamma) \geq \intersection(h^k(\gamma), \gamma). \]

To compute these intersection numbers we use the algorithm of Schaefer, Sedgwick and \v{S}tefankovi\v{c} \cite[Section~5]{Schaefer07computingdehn}. This computes the intersection number of two multicurves from their edge vectors and does so in polynomial time \cite[Lemma~5.4]{Schaefer07computingdehn}.

Thus given the edge vectors of a $\gamma$ and $\delta$, we can use the algorithm of Schaefer, Sedgwick and \v{S}tefankovi\v{c} to first compute $\intersection(\gamma, \gamma_i)$ and so verify that $\gamma$ is dual to $\gamma_i$. Furthermore we can crush $S$ along $\delta$ and check that the component of $S_\delta$ containing $\gamma_i$ and $\gamma$ is either a four times marked sphere or a once-marked torus, depending on whether $\gamma_i$ is separating or not. This verifies that $\delta = \partial N(\gamma_i \cup \gamma)$. Having established that $\gamma$ and $\delta$ are the relevant curves, we can use them to compute $d_i$ by the previous formulae. 

Now as $\calT(\gamma_i)$ is $O(\ell(p))$--bounded there is a dual curve $\gamma$ such that $\calT(\gamma)$ is also $O(\ell(p))$--bounded. For such a $\gamma$, the edge vector of the corresponding $\delta$ is also $O(\ell(p))$--bounded. Therefore these intersection number and $S_\delta$ can all be computed in $O(\poly(\ell(p)))$ operations by the algorithm of Schaefer, Sedgwick and \v{S}tefankovi\v{c} and Theorem~\ref{thrm:crush_complexity}. Hence we can compute $d_i$ in $O(\poly(\ell(p)))$ operations too.

\subsubsection{Equivalence of partition graphs}
\label{subsub:equiv_partition_graphs}

Now suppose that $(G, \phi)$ and $(H, \psi)$ are the partition graphs of $g$ and $h$ respectively. A graph isomorphism $\Phi \from G \to H$ induces a bijection between the marked points of $g_i$ coming from $\canonical(g)$ and those of $\Phi(g_i)$ coming from $\canonical(h)$. We denote this by $D_{g_i} \Phi$. 

\begin{theorem}[{\cite[Theorem~8.3]{MosherLecture} \cite[Section~1.14]{MosherArationality}}]
Suppose that $(G, \phi)$ and $(H, \psi)$ are the partition graphs of $g$ and $h$ respectively. Then $g$ and $h$ are conjugate if and only if there is a graph isomorphism $\Phi \from G \to H$ such that:
\begin{enumerate}
\item $\Phi$ conjugates $\phi$ to $\psi$,
\item $g_i$ and $\Phi(g_i)$ are $D_{g_i} \Phi$--conjugate, and
\item $\gamma_j$ and $\Phi(\gamma_j)$ have the same twist invariant. \qed
\end{enumerate}
\end{theorem}

Thus when $g$ and $h$ are not conjugate, for each graph isomorphism $\Phi \from G \to H$ either:
\begin{enumerate}
\item $\Phi \circ \phi \neq \psi \circ \Phi$, which we can check in $\zeta \in O(1)$ operations,
\item there is a vertex $g_i \in G$ such that $g_i$ and $\Phi(g_i)$ are not $D_{g_i} \Phi$--conjugate and so by Corollary~\ref{cor:p_conjugate} and Corollary~\ref{cor:ap_ir_conjugate} we can provide a polynomial-time checkable proof of this fact, or
\item there is an edge $\gamma_j \in G$ such that $\gamma_j$ and $\Phi(\gamma_j)$ have different twist invariants and so by Lemma~\ref{subsub:twist_invariants} we can provide a polynomial-time checkable proof of this fact.
\end{enumerate}
By doing this for each of the at most $\zeta ! \in O(1)$ such isomorphisms, we can provide a proof that $g$ and $h$ are not conjugate which can be verified in $O(\poly(\ell(p) + \ell(q)))$ operations.

Again, applying this result to the standard path of a word we obtain that:
\begin{corollary}
\label{cor:ap_r_conjugate}
Deciding whether two aperiodic, reducible words $g, h \in X^*$ correspond to conjugate mapping classes is a problem in $\coNP$. \qed
\end{corollary}

Together with Corollary~\ref{cor:p_conjugate} and Corollary~\ref{cor:ap_ir_conjugate} this shows that:

\begin{corollary}
\label{cor:conjugacy_coNP}
Fix $S$, a marked surface, and $X$, a finite generating set of $\Mod(S)$. Deciding whether two words $g, h \in X^*$ correspond to conjugate mapping classes is a problem in $\coNP$. \qed
\end{corollary}

Finally, we note that again this immediately gives an exponential time algorithm to determine whether two words represent conjugate mapping classes.

\section{Implementation}
\label{sec:implementation}

The main algorithm has been implemented as part of \texttt{flipper} \cite{flipper}.

However the implementation of the main algorithm used by \texttt{flipper} has some important differences. These stem from the fact that the constant $K$ of Theorem~\ref{thrm:linear_conjugator_bound} is unknown. Therefore to decide whether a given lamination is filling \texttt{flipper} computes the splitting sequence $T_0, T_1, \ldots$ until it finds a pair which are projectively equal. Hence \texttt{flipper} will actually perform $t + m$ splits, although this is still $O(\ell(h)^2)$. As the total number of steps needed is unknown at the beginning, \texttt{flipper} records the weights of $T_i$ algebraically as integer linear combinations of the weights of $T_0$. It is possible that the sequence of splits continues for so long that the precision to which these weights are known no longer uniquely determines them. In this case $T_0$ is recalculated to higher precision using an iterative technique and the computation can continue.

An implementation of the solution to the conjugacy problem in the pseudo-Anosov case is also included in \texttt{flipper}.

\bibliographystyle{plain}
\bibliography{bibliography}

\end{document}